\documentclass[reqno,11pt]
{amsart}

\usepackage{url}
\usepackage{latexsym}
\usepackage{lscape} 
\usepackage{pdflscape}
\usepackage{float}
\usepackage{amssymb}
\usepackage{amscd}

\usepackage{tikz,tikz-cd}
\usetikzlibrary{arrows,positioning,automata,shadows,fit,shapes}

\usepackage{faktor}
\usepackage{algorithm}
\usepackage{algpseudocode}

\usepackage[pdftex,
unicode=true,
colorlinks=false
]{hyperref}

\usepackage{subfiles} 

\newtheorem{thm}[equation]{Theorem}
\newtheorem{lem}[equation]{Lemma}

\theoremstyle{remark}
\newtheorem{rem}[equation]{Remark}

\newtheorem{ex}[equation]{Example}

\theoremstyle{definition}

\numberwithin{equation}{section}

\newcommand{\End}{\operatorname{End}}
\newcommand{\Hom}{\operatorname{Hom}}

\newcommand{\fg}{{\mathfrak g}}

\newcommand{\fk}{{\mathfrak k}}

\newcommand{\fp}{{\mathfrak p}}

\newcommand{\ft}{{\mathfrak t}}

\newcommand{\be}{\begin{equation}}
\newcommand{\beu}{\begin{equation*}}

\newcommand{\bbar}{\,\big|\,}

\newcommand{\pr}{\operatorname{pr}}

\usepackage{graphicx,color,cancel}

\newcommand\la{{\lambda}}

\newcommand\al{{\alpha}}

\newcommand{\twedge}{{\textstyle{\bigwedge}}}

\newcommand\GL{\operatorname{GL}}

\newcommand\gr{\operatorname{gr}}
\newcommand\Gr{\operatorname{Gr}}

\newcommand\id{\operatorname{id}}

\newcommand{\pf}{\begin{proof}}
\newcommand{\epf}{\end{proof}}
\newcommand{\eq}{\begin{equation}}
\newcommand{\eeq}{\end{equation}}
\newcommand{\eqn}{\begin{equation*}}
\newcommand{\eeqn}{\end{equation*}}

\newcommand{\frg}{\mathfrak{g}}

\newcommand{\frk}{\mathfrak{k}}

\newcommand{\frp}{\mathfrak{p}}

\newcommand{\frt}{\mathfrak{t}}

\newcommand{\bbC}{\mathbb{C}}

\newcommand{\smat}{\left(\begin{smallmatrix}}
\newcommand{\esmat}{\end{smallmatrix}\right)}
\newcommand{\pmat}{\begin{pmatrix}}
\newcommand{\epmat}{\end{pmatrix}}

\newcommand{\cP}{\mathcal{P}}

\newcommand{\fgl}{\mathfrak{gl}}

\DeclareMathOperator{\Cl}{Cl}

\DeclareMathOperator{\ev}{ev}

\title[Clifford algebras and Littlewood-Richardson coefficients]{Clifford algebras and Littlewood-Richardson coefficients}

\author{Kieran Calvert}
\address[K.~Calvert]{School of Mathematical Sciences, Lancaster University, Bailrigg, Lancaster, UK}
\email{kieran.calvert@lancaster.ac.uk}

\author{Karmen Grizelj}
\address[K.~Grizelj]{Department of Mathematics, Faculty of Science, University of Zagreb, Bijeni\v{c}ka cesta 30, 10000 Zagreb, Croatia}
\email{karmen.grizelj@math.hr}

\author{Andrey Krutov}
\address[A.~Krutov]{Mathematical Institute of Charles University, Sokolovsk\'a 83, 186 75 Prague, Czech Republic} 
\email{andrey.krutov@matfyz.cuni.cz}

\author{Pavle Pand\v{z}i\'c}
\address[P.~Pand\v{z}i\'c]{Department of Mathematics, Faculty of Science, University of Zagreb, Bijeni\v{c}ka cesta 30, 10000 Zagreb, Croatia}
\email{pandzic@math.hr}

\thanks{
  K.~Grizelj and P.~Pand\v{z}i\'{c} were supported by the project ``Implementation of cutting-edge research and its application as part of the Scientific Center of Excellence for Quantum and Complex Systems, and Representations of Lie Algebras'', PK.1.1.02, European Union, European Regional Development Fund.
  A.~Krutov was supported by the GA\v{C}R project 24-10887S and HORIZON-MSCA-2022-SE-01-01 CaLIGOLA.
  This article is based upon work from COST Action CaLISTA CA21109 supported by COST (European
  Cooperation in Science and Technology, \url{www.cost.eu}).
}



\begin{document}

\begin{abstract}
We show how to use Clifford algebra techniques to describe the de Rham cohomology ring of equal rank compact symmetric spaces $G/K$. In particular, for $G/K=U(n)/U(k)\times U(n-k)$, we obtain a new way of multiplying Schur polynomials, i.e., computing the Littlewood-Richardson coefficients. 
The corresponding multiplication on the Clifford algebra side is, in a convenient basis given by projections of the spin module, simply the componentwise multiplication of vectors in $\bbC^N$, also known as the Hadamard product.
\end{abstract}

\sloppy

\maketitle

\section{Introduction}

In this paper we consider the classical compact symmetric space $G/K=U(n)/U(k)\times U(n-k)$. It is well known that this space is diffeomorphic to the complex Grassmannian 
$\Gr(k,\bbC^n)$ of $k$-dimensional subspaces of $\bbC^n$. The de Rham cohomology $H(G/K)$ (with complex coefficients) is very well understood through Schubert calculus. 
One of the well known results about $H(G/K)$ (due to \'E.Cartan and de Rham) is that it is isomorphic to $(\twedge\frp^*)^\frk$, where $\frg=\frk\oplus\frp$ is the decomposition of the complexified Lie algebra of $G$ into eigenspaces of the (differentiated) involution defining $K$. By classical results of Borel, H.~Cartan and others, 
\eq\label{thmCB}
  \bbC[\frt^*]^{W_\frk}/ I_+ \cong (\twedge\frp^*)^\frk.
\eeq
Here $\frt$ is a~Cartan subalgebra of~$\frk$ and $\frg$, $W_\frk$ is the Weyl group of $(\frk,\frt)$, and $I_+$ is the ideal in $\bbC[\frt^*]^{W_\frk}$ generated by the  polynomials without constant term invariant under the Weyl group $W_\frg$ of $(\frg,\frt)$.

Furthermore, there is a~basis for $(\twedge\frp^*)^\frk$ given by the images under the isomorphism~\eqref{thmCB} of Schur polynomials $s_\la\in\bbC[\frt^*]^{W_\frk}$   associated to partitions~$\la$ with Young diagrams inside the $k\times (n-k)$ box. 
For a~definition of $s_\la$ see e.g.~\cite[Chapter 1]{mac}. 

The multiplication of Schur polynomials is described by the Littlewoood-Richardson rule~\cite{LittlewoodRichardson34}:  
\[ 
s_\lambda s_\mu = \sum_{\nu} c_{\lambda \mu}^\nu s_{\nu},
\]
where $c_{\lambda \mu}^\nu$ are the Littlewood-Richardson (LR) coefficients.  
The LR coefficients also appear in other places. For example, let $V_\la$, $V_\mu$ and $V_\nu$ be finite-dimensional representations of $\GL_m(\bbC)$ with highest weights given by partitions $\lambda,\mu,\nu$. Then one has
\[
   \dim \Hom_{GL_m(\mathbb{C})}(V_\nu,V_\lambda\otimes V_\mu)= c^{\nu}_{\lambda\mu}.
\]
The LR coefficients also appear in branching laws as well as in Schubert calculus.

The multiplication of the basis of Schur polynomials in $H(\Gr(k,\bbC^n))$ is again given by the Littlewood-Richardson rule, but we replace
$c_{\lambda \mu}^\nu$ by zero if $\nu$ does not fit in the $ k \times (n-k)$  box; for example, see~\cite{Lesieur1947, Borel1953, BGG1973} and~\cite[Theorem 3.1]{Stanley76}.

In the following, we denote the subspace of  $\bbC[\frt^*]^{W_\frk}$ spanned by the above Schur polynomials by $\cP_{k\times (n-k)}$.

In~\cite{CNP} and~\cite{CGKP} we study a filtered deformation of $(\twedge\frp^*)^\frk\cong (\twedge\frp)^\frk$, given as the $\fk$-invariants $\Cl(\frp)^\frk$ in the Clifford algebra $\Cl(\frp)$ with respect to the trace form~$B$. 
Recall that $\Cl(\frp)$ is defined as the quotient of the tensor algebra $T(\frp)$ by the ideal generated by 
\[
X\otimes Y+Y\otimes X-2B(X,Y),\qquad X,Y\in\frp.
\]
The grading by degree of $T(\frp)$ induces a~filtration of $\Cl(\frp)$, and the associated graded algebra of the filtered algebra $\Cl(\frp)$ is $\twedge\frp$. Passing to $\frk$-invariants, we conclude that $\Cl(\frp)^\frk$ is a~filtered algebra with the associated graded algebra equal to $(\twedge\frp)^\frk$.

The main advantage of $\Cl(\frp)^\frk$ over $(\twedge\frp)^\frk$ is that the algebra structure is much simpler. Namely, $\Cl(\frp)$ is the endomorphism algebra of the spin module $S$, which becomes a~$\frk$-module through the map $\alpha:U(\frk)\to \Cl(\frp)$ defined on $X\in\frk$  by
\eq\label{alpha}
\alpha(X)=\frac{1}{4}\sum_i [X,b_i]d_i,
\eeq
where $b_i$ is any basis of~$\frp$ and~$d_i$ is the dual basis with respect to~$B$. One knows that the $\frk$-module $S$ is multiplicity free, with $|W_\frg|/|W_\frk|=\binom{n}{k}$ components. 
In fact, these components have highest weights $\sigma\rho-\rho_\frk$, where~$\rho$ and~$\rho_\frk$ are the half sums of compatible positive root systems for $(\frg,\frt)$ respectively $(\frk,\frt)$, and $\sigma$ runs over the set 
\eq\label{w1}
W^1=\{\sigma\in W_\frg\bbar \sigma\rho \text{ is $\frk$-dominant}\}.
\eeq
It now follows from Schur's lemma that $\Cl(\frp)^\frk=\End_\frk S$ is spanned by the orthogonal projections to these components. In this way the algebra $\Cl(\frp)^\frk$ becomes isomorphic to $\bbC^{\binom{n}{k}}$ with componentwise multiplication, also called the Hadamard product.

The following filtered analogue of the Cartan--\/Borel theorem is proved in~\cite[Corollary 3.15]{CNP}; we give a slightly more direct argument in Section~\ref{s:prelims}. A~version of this theorem when~$\frg$ and~$\frk$ are not of equal rank will appear in~\cite{CGKP}.

\begin{thm}
\label{thm:structure} The algebra homomorphism $\alpha$ of~\eqref{alpha} restricted to $U(\frk)^\frk\cong \bbC[\frt^*]^{W_\frk}$ induces a~filtered algebra isomorphism 
\[
\bbC[\frt^*]^{W_\frk} /I_\rho\cong \Cl(\frp)^\frk,
\]
where $I_\rho$ is the ideal in $\bbC[\frt^*]^{W_\frk}$ generated by the $W_\frg$-invariant polynomials vanishing at~$\rho$.
\end{thm}

The Cartan--\/Borel theorem~\eqref{thmCB} can be deduced from Theorem~\ref{thm:structure} by passing to associated graded algebras. 
In particular, the isomorphism~\eqref{thmCB} is induced by~$\gr\alpha$.
    
Through the above identifications, (images of) Schur polynomials define a~basis for $\Cl(\frp)^\frk$, and since the multiplication in $\Cl(\frp)^\frk$ is equal to the multiplication in $(\twedge\frp^*)^\frk$ modulo lower order terms, we can recover the Littlewood-Richardson coefficients from the multiplication in $\Cl(\frp)^\frk$, which is simply the Hadamard product in the basis of projections. Following Kostant, we view~$\twedge\frp$ and~$\Cl(\frp)$ as the same space with two multiplications, $\wedge$ respectively~$\bullet$. In particular, this space is graded, and we denote by~$a_{[d]}$ the $d$th graded component of~$a$ in this space. (We point out that while~$\wedge$ is compatible with this grading, $\bullet$ is not.)

The main result of this paper is
\begin{thm}\label{thm:main}
  Given Schur polynomials~$s_\lambda$ and~$s_\mu$ of degree $l(\lambda)$ respectively $l(\mu)$, 
  we can calculate the wedge product  of their images  
  in $H^*(\Gr(k,\mathbb{C}^{n}))\cong \bbC[\ft^*]^{W_\frk}/I_+$ as
  \[
    \gr\al(s_\lambda)\wedge  \gr\al(s_\mu) = (\al(s_\lambda)\bullet_H  \al(s_\mu))_{[2l(\lambda)+2l(\mu)]}, 
  \]
This describes a~novel way to calculate the Littlewood-Richardson coefficients using the~Hadamard product (componentwise multiplication) $\bullet_H$, a~linear isomorphism from the space $\cP_{k \times (n-k)}$ of Schur polynomials to $\Cl(\frp)^\frk$ induced by $\alpha$, and the inverse of this isomorphism. Let $M$ be the matrix sending the basis $s_\lambda  \in
\cP_{n \times (n-k)}$ to $ \alpha(s_\lambda) \in \mathrm{Span}\{ \pr_\sigma\}$, and $P_{[d]}$ be the projection to degree $d$ polynomials.
Then the multiplication of Schur polynomials~$s_\lambda$ and~$s_\mu$ in the cohomology ring $H(\Gr(k,\bbC^n)) \cong \twedge(\frp)^\frk$ is 
\[   s_\lambda \cdot s_\mu 
 =P_{[l(\lambda) +  l (\mu)]}M^{-1}(M(s_\lambda) \bullet_H M(s_\mu)).\]
\end{thm}

The above is illustrated by Algorithm~\ref{a:alg1} and examples in Section~\ref{s:apptoLR}. Furthermore, the linear isomorphism corresponding to~$M$ is easily described as evaluation of Schur polynomials at points~$\sigma \rho$ for $\sigma \in W^1$, see~\eqref{evrho}.

To obtain the full list of LR coefficients for a~given $\lambda$ and~$\mu$, one can
consider the case of $n = 2(l(\lambda) + l(\mu))$, $k=l(\lambda)+l(\mu)$. But in order to compute
$c_{\lambda\mu}^\nu$ for given~$\lambda$, $\mu$, and~$\nu$ it is enough to consider, for example, the minimal box $\cP_{k\times (n-k)}$ such that
$\lambda,\mu,\nu\in\cP_{k\times (n-k)}$.

\begin{algorithm}[H]
\caption{An algorithm to calculate $s_\lambda \cdot s_\mu$}\label{a:alg1}
\begin{algorithmic}
\Require $\lambda,\mu \in \cP_{k\times (n-k)}$
\Ensure the list of $c_{\lambda\mu}^\nu$ for $\nu\in\cP_{k\times(n-k)}$
\State $\mathrm{Poly} =\{s_\nu: \nu \in \cP_{k \times (n-k)}\}$, $N = {n \choose k}$, $\mathrm{Pts} = \{\}$.
\State $E_\lambda \gets $ the coordinate vector  of $s_\lambda$ in the basis $\mathrm{Poly}$
\State $E_\mu \gets $ the coordinate vector of $s_\mu$ in the basis $\mathrm{Poly}$
\State $\rho \gets \{ \frac{n+1}{2}, \frac{n+1}{2}-1, \cdots, \frac{n+1}{2} - (n-1), \frac{n+1}{2} - (n-1)\}$
\Comment{Calculate~$\rho$}
\For{$\sigma$ a $k$ shuffle of $n$}
\State $\mathrm{Pts} \gets \mathrm{Pts} \cup \sigma(\rho)$ \Comment{Calculate the points $\{ \sigma \rho : \sigma \in W^1\}$}
\EndFor
\For{$i = 1,\cdots, N$}
\For{$j = 1, \cdots, N$}
    \State $M[i,j] \gets \mathrm{Poly}[i](\mathrm{Pts}[j])$ \Comment{$M$ is the basis change $\{s_\lambda\}$ to $\{\pr_{\sigma}\}$}
\EndFor
\EndFor
\State $X_\lambda \gets M \cdot E_\lambda$, $X_\mu \gets M\cdot E_\mu$ \Comment{Coordinates of $s_\lambda$,$s_\mu$ in the basis of $\{\pr_\sigma\}$}
\For{$i =1,\cdots,N$} 
\State $W[i] \gets X_\lambda[i]\times X_\mu[i]$ \Comment{Hadamard product}
\EndFor
\State $Z \gets M^{-1}\cdot W$ \Comment{Back to the basis of $\{s_\lambda\}$}
\For{$ i = 1,\cdots, N$} 
\If{$\deg \mathrm{Poly}[i] \neq \deg s_\lambda +\deg s_\mu$}
    \State $Z[i] \gets 0$ \Comment{Removing lower order terms}
\EndIf
\EndFor
\State \Return $Z$ \Comment{The list of $c_{\lambda\mu}^\nu$}
\end{algorithmic}
\end{algorithm}
In Section~\ref{s:prelims} we discuss the $\frk$-decomposition of the spin module~$S$ and prove Theorems~\ref{thm:structure} and~\ref{thm:main}. In Section~\ref{s:apptoLR} we compute a~couple of examples, calculated using~\cite[Mathematica]{Mathematica141}.

\section{Description of $\Cl(\frp)^\frk$ for equal rank symmetric pairs}\label{s:prelims}

In this section~$G$ and~$K$ are as in the introduction, but $(G,K)$ could equally well be any equal rank compact symmetric pair with~$G$ and~$K$ connected (for disconnected $K$, see~\cite[Section 3]{CNP}). Since~$G$ and~$K$ are connected it is enough to work with Lie algebras~$\frg$ and~$\frk$. In this more general situation, the form~$B$ is replaced by any nondegenerate invariant symmetric bilinear form on~$\frg$, and~$\frp$ is still the $-1$ eigenspace of the involution defining~$\frk$ (or equivalently, the orthogonal of~$\frk$ with respect to~$B$).

Let us first review some facts about the spin module~$S$ for $\Cl(\frp)$. 
Since $\dim\frp$ is even, the Clifford algebra $\Cl(\frp)$ has only one simple module~$S$, and moreover $\Cl(\frp)\cong\End S$. To construct~$S$, write
\[
\frp=\frp^+\oplus\frp^-,
\]
where $\frp^+$ and $\frp^-$ are maximal isotropic subspaces of~$\frp$ in duality under~$B$. Then one can take $S=\twedge\frp^+$, with elements of~$\frp^+$ acting by wedging, and elements of~$\frp^-$ by contracting. 
As mentioned in the introduction, $S$ becomes a~$\frk$-module through the map~$\alpha$ of~\eqref{alpha}. Furthermore, the $\frk$-module~$S$ is multiplicity free and decomposes as
\[
S=\bigoplus_{\sigma\in W^1} E_{\sigma\rho-\rho_\frk},
\]
where~$\rho$, $\rho_\frk$ and $W^1$ are as in the introduction, see~\eqref{w1},
and $E_{\sigma\rho-\rho_\frk}$ denotes the irreducible finite-dimensional $\frk$-module with highest weight $\sigma\rho-\rho_\frk$.

Now $\Cl(\frp)\cong\End S$ implies 
\[
\Cl(\frp)^\frk=\End_\frk S,
\]
and by Schur's Lemma this is the algebra of projections $\pr_\sigma: S\to E_{\sigma\rho-\rho_\frk}$, which we denote by $\Pr(S)$. This is a~very simple commutative algebra, isomorphic to $\bbC^{|W^1|}$ with coordinatewise multiplication (Hadamard product), by identifying~$\pr_\sigma$ with $(0,\dots,0,1,0,\dots,0)$, with~$1$ in the place corresponding to~$\sigma$.

The restriction of~$\al$ to the center $U(\frk)^\frk$ of $U(\frk)$ has a~very simple description in terms of the Harish-Chandra isomorphism $U(\frk)^\frk\cong\bbC[\frt^*]^{W_\frk}$. Namely, under the Harish-Chandra isomorphism, the infinitesimal character of $E_{\sigma\rho-\rho_\frk}$ corresponds to evaluation on $\sigma\rho$. Therefore, the restriction of $\alpha$ to $U(\frk)^\frk$ corresponds to the 
algebra morphism, $ \ev_\rho : \bbC[\ft^*]^{W_\frk} \to \Cl(\frp)^\frk =\Pr(S)$ 
given by 
\eq \label{evrho}
\ev_\rho(p) = \sum_{\sigma \in W^1} p(\sigma \rho) \pr_\sigma.
\eeq

\begin{lem}\label{l::alphaSurjective}
  The map $\ev_\rho$ is surjective. 
\end{lem}

\begin{proof}
We need to prove that for any choice of scalars $a_\sigma$, $\sigma\in W^1$, there is  $P \in \bbC[\frt^*]^{W_\frk}$ such that 
$P(\sigma\rho)=a_\sigma$ for all $\sigma\in W^1$. The polynomial $P$ is similar to the Lagrange interpolation polynomial, but we have to ensure it is $W_\frk$-invariant.

It is enough to find the polynomials $P_\sigma \in \bbC[\frt^*]^{W_\frk}$, $\sigma\in W^1$, such that $P_\sigma(\tau\rho) = \delta_{\sigma\tau}$ for any~$\tau\in W^1$. 

We define
\[
  Q_\sigma (\gamma) = \prod_{\substack{w\in W_\frk \\ w\neq id \\ \tau\in W^1}}^{} \frac{||\gamma - w \tau \rho||^2}{||\sigma \rho - w \tau \rho||^2} \prod_{\tau \neq \sigma}^{} \frac{||\gamma - \tau \rho||^2}{||\sigma \rho - \tau \rho||^2},\qquad \gamma\in\frt^*.
  \]    
Then
\begin{align*}
  &Q_\sigma (\sigma \rho) = 1, \\
  &Q_\sigma (w \tau \rho) = 0 \quad \text{ if } w\in W_\frk\setminus\{\id\}  \text{ or } \sigma\neq\tau.
\end{align*}
  
So $Q_\sigma(\tau\rho) = \delta_{\sigma\tau}$ as required, but the $Q_\sigma$ are not necessarily $W_\frk$-invariant. So we set
\[
  P_\sigma = \frac{1}{|W_\frk|} \sum_{w \in W_\frk} w Q_\sigma
\]
and obtain the polynomials we were looking for.
\end{proof}

The kernel of~$\ev_\rho$ clearly contains the ideal~$I_\rho$ generated by $\{ p \in \bbC[\frt^*]^{W_\frg}: p (\rho) = 0\}$. To see that in fact $\ker\ev_\rho=I_\rho$, we prove

\begin{lem}\label{l::Icodim}
    The ideal $I_\rho$ is of codimension $|W^1|$ in $\bbC[\frt^*]^{W_\frk}$.
\end{lem}
\begin{proof}
    Let~$J$ be the ideal of~$\bbC[\frt^*]$ generated by $\{ p \in \bbC[\frt^*]^{W_\frg}: p (0) = 0\}$, 
 then $\bbC[\frt^*]/J$ is the coinvariant algebra associated to $W_\frg$ and is isomorphic as a $W_\frg$-module to $\bbC[W_\frg]$, see~\cite[Section 3.6]{HumphreysRefBook}, hence has dimension~$|W_\frg|$. Therefore $J$ has codimension $|W_\frg|$.
    The associated graded ideal of~$I_\rho$ is~$I_+$. Applying $W_\frk$ invariance to the exact sequence of $W_\frg$-modules 
    \[ 0 \to J \to \bbC[\frt^*] \to \bbC[\frt^*]/J \cong \bbC[W_\frg] \to 0 \]
    we obtain the exact sequence 
     \[ 0 \to J^{W_\frk} \cong I_+ \to \bbC[\frt^*]^{W_\frk} \to (\bbC[\frt^*]/J)^{W_\frk} \cong \bbC[W_\frg]^{W_\frk} \to 0. \]
     We can conclude that $I_+ = J^{W_\frk}$ is of codimension $\dim \bbC[W_\frg]^{W_\frk} = |W^1|$ and hence, since the functor~$\gr$ preserves codimension, so is~$I_\rho$.
\end{proof}

\begin{proof}[Proof of Theorem \ref{thm:structure}] We have seen that the map $\alpha:U(\frk)^\frk\to \Cl(\frp)^\frk$ becomes $\ev_\rho$ under identifications $U(\frk)^\frk=\bbC[\frt^*]^{W_\frk}$ and $\Cl(\frp)^\frk=\Pr(S)$. By Lemma~\ref{l::alphaSurjective}, $\ev_\rho$ is onto, and by Lemma~\ref{l::Icodim}, its kernel is exactly~$I_\rho$. 
\end{proof}

The above discussion can be summarized by the following commutative diagram, with bottom row exact:
\[
  \begin{tikzcd}
    &&U(\frk)^\frk \arrow[r, "\alpha"] \arrow[d, "\mathrm{hc}"]& \Cl(\frp)^\frk \arrow[d, "\cong"]\\
    0\arrow[r] & I_\rho \arrow[r] &\bbC[\frt^*]^{W_\frk}\arrow[r, "\ev_\rho"] & \Pr(S) \arrow[r] & 0
  \end{tikzcd}
\]

\begin{rem} The above proof of Theorem \ref{thm:structure} works equally well for an unequal rank symmetric pair $(\frg,\frk)$. In this case, the statement is that $\alpha$ (or $\ev_\rho$) maps $U(\frk)^\frk\cong\bbC[\frt^*]^{W_\frk}$ onto $\Pr(S)$, which is however no longer all of $\Cl(\frp)^\frk$, and that the kernel of $\alpha$ is  generated by the polynomials invariant under the Weyl group of the pair $(\frg,\frt)$. Here $\frt$ is a Cartan subalgebra of $\frk$, which is no longer a Cartan subalgebra of $\frg$. 

In the more general case when $\frk$ is a reductive quadratic subalgebra of $\frg$ which is not necessarily symmetric, the above proof still shows that $\alpha$ maps $U(\frk)^\frk$ onto $\Pr(S)$, but it is more difficult to describe its kernel.
\end{rem}

\begin{proof}[Proof of Theorem \ref{thm:main}]
  Recall that we view~$\Cl(\frp)$ and~$\twedge\frp$ as the same space with two multiplications, $\bullet$ and~$\wedge$. Note that for elements~$a$ and~$b$ of 
  this space such that~$a$ (resp.~$b$) is in the $k$th (resp. $l$th) filtered piece but not the $(k-1)$st (resp. $(l-1)$st) filtered piece, we have  
  \[
    a_{[k]} \wedge b_{[l]} = (a \bullet b)_{[k+l]}.
  \]
  Also, $\gr \alpha(s_\lambda) =  \alpha(s_\lambda)_{[2l(\lambda)]}$, hence
  \begin{multline*}
    \gr \alpha( s_\lambda)\wedge \gr \alpha( s_\mu)  =\alpha(s_\lambda)_{[2l(\lambda)]} \wedge \alpha(s_\mu)_{[2l(\mu)]}  = (\alpha(s_\lambda) \bullet \alpha(s_\mu))_{[2l(\lambda) +2l(\mu)]}.
  \end{multline*}
\end{proof}

\section{Examples} \label{s:apptoLR}
We highlight the application of Theorem~\ref{thm:main} to the calculation of Littlewood-Richardson coefficients below with a couple of examples.

\begin{ex}
For $\fg = \fgl_4$ and $\fk = \fgl_2 \oplus \fgl_2$, the set $\sigma \rho$ for $\sigma\in W^1$ is
\begin{gather*}
  (\tfrac{3}{2},\, \tfrac{1}{2},\, -\tfrac{1}{2},\, -\tfrac{3}{2}),\, \qquad (\tfrac{3}{2},\, -\tfrac{1}{2},\, \tfrac{1}{2},\, -\tfrac{3}{2}),\, \qquad (\tfrac{3}{2},\, -\tfrac{3}{2},\, \tfrac{1}{2},\, -\tfrac{1}{2}),\, \\
  (\tfrac{1}{2},\, -\tfrac{1}{2},\, \tfrac{3}{2},\, -\tfrac{3}{2}),\, \qquad (\tfrac{1}{2},\, -\tfrac{3}{2},\, \tfrac{3}{2},\, -\tfrac{1}{2}),\, \qquad (-\tfrac{1}{2},\, -\tfrac{3}{2},\, \tfrac{3}{2},\, \tfrac{1}{2})
\end{gather*}
Let $\{\pr_i: i = 1,\ldots,6\} = \{\pr_\sigma: \sigma \in W^1\}$ be the basis of orthogonal projections in $\Pr(S)$. By~\eqref{evrho},
\[
  \alpha(s_\lambda) = \ev_\rho(s_\lambda) = \sum_{\sigma\in W^1} s_\lambda(\sigma\rho)\pr_\sigma, \quad s_\lambda \in \cP_{k \times (n-k)} .  
\]
For simplicity of notation, for the remainder of these examples we will replace~$\alpha(s_\lambda)$ by~$s_\lambda$. Expressing $s_\lambda$ as a~vector in $\bbC^6 \cong \Pr(S)$ 
\begin{align*}
  s_{(0,0)} ={}& (1,1,1,1,1,1),&
  s_{(1,0)} ={}& (2,1,0,0,-1,-2) ,\\
  s_{(2,0)} ={}&( \tfrac{13}{4} ,\tfrac{7}{4} , \tfrac{9}{4} , \tfrac{1}{4},\tfrac{7}{4},\tfrac{13}{4}),&  
  s_{(1,1)} ={}& (\tfrac{3}{4} - \tfrac{3}{4} ,- \tfrac{9}{4} , - \tfrac{1}{4}, - \tfrac{3}{4} ,\tfrac{3}{4}),\\
  s_{(2,1)} ={}& (\tfrac{3}{2} ,  - \tfrac{3}{4} , 0,0, \tfrac{3}{4} , - \tfrac{3}{2}),&
  s_{(2,2)} ={}&( \tfrac{9}{16}, \tfrac{9}{16}, \tfrac{81}{16} , \tfrac{1}{16} , \tfrac{9}{16}, \tfrac{9}{16}).
\end{align*}
The change of basis from $s_\lambda$ to $\pr_i$ and its inverse are given by the matrices
\[M =\begin{bmatrix}
1 & 1 & 1 & 1 & 1 & 1 \\
2 & 1 & 0 & 0 & -1 & -2 \\
\tfrac{13}{4} & \tfrac{7}{4} & \tfrac{9}{4} & \tfrac{1}{4} &\tfrac{7}{4} & \tfrac{13}{4} \\[3pt]
\tfrac{3}{2} & \tfrac{-3}{4} & 0& 0&\tfrac{3}{4} & \tfrac{-3}{2} \\[3pt]
\tfrac{9}{16} & \tfrac{9}{16} & \tfrac{81}{16} & \tfrac{1}{16} & \tfrac{9}{16} & \tfrac{9}{16} \\
    
\end{bmatrix}, \ M^{-1} =\begin{bmatrix}
\tfrac{3}{64} & \tfrac{1}{8} & \tfrac{1}{16} & \tfrac{13}{48} & \tfrac{1}{ 6} & \tfrac{1}{12} \\[3pt]
\tfrac{-3}{16} & \tfrac{1}{4} & \tfrac{1}{4} & \tfrac{-7}{12} & \tfrac{-1}{3} & \tfrac{-1}{3} \\[3pt]
\tfrac{1}{16} & 0& \tfrac{-1}{16} & \tfrac{1}{16} & 0 & \tfrac{1}{4} \\[3pt]
\tfrac{-3}{16} & \tfrac{-1}{4} & \tfrac{1}{4}& \tfrac{-7}{12}&\tfrac{1}{3} & \tfrac{-1}{3} \\[3pt]
\tfrac{3}{64} & \tfrac{-1}{8} & \tfrac{1}{16} & \tfrac{13}{48} & \tfrac{-1}{6} & \tfrac{1}{12} \\[3pt]
    
\end{bmatrix}.\] 

The multiplication~$\bullet$ in $\Cl(\frp)^\frk$ is given by expressing~$s_\lambda$ and~$s_\mu$ as vectors in $\bbC^6 = \mathrm{Span}_\bbC(\pr_i: i =1,\ldots ,6)$ using~$M$, then multiplying with the Hadamard product. The resulting vector is then expressed as a~combination of~$s_\lambda$ by applying~$M^{-1}$. Explicitly $s_\lambda \bullet s_\mu = M^{-1}(M(s_\lambda) \bullet_H M(s_\mu))$, where~$\bullet_H$ is the Hadamard (component-wise) product. This gives the new multiplication below: (The terms vanishing in the associated graded algebra~$\twedge(\fp)^\fk$, equivalently those vanishing under the projection $P_{[l(\lambda) +  l (\mu)]}$, are underlined.) 
{\small\begin{align*}\allowbreak
  s_{(1,0)}\bullet s_{(1,0)}={}& s_{(1,1)}+s_{(2,0)},& 
  s_{(1,0)}\bullet s_{(2,0)}={}& \underline{\tfrac{5}{2} s_{(1,0)}}+s_{(2,1)},\\
  s_{(1,0)}\bullet s_{(1,1)}={}& s_{(2,1)},& 
  s_{(1,0)}\bullet s_{(2,1)}={}& \underline{\tfrac{9}{16} s_{(0,0)}+\tfrac{5}{2} s_{(1,1)}}+s_{(2,2)},\\
  s_{(1,0)}\bullet s_{(2,2)}={}& \underline{\tfrac{9}{16} s_{(1,0)}},& 
  s_{(2,0)}\bullet s_{(2,0)}={}& \underline{\tfrac{5}{2} s_{(1,1)}+\tfrac{5}{2} s_{(2,0)}}+s_{(2,2)},\\
  s_{(1,1)}\bullet s_{(2,0)}={}& \underline{\tfrac{9}{16} s_{(0,0)}+\tfrac{5}{2} s_{(1,1)}},& 
  s_{(2,0)}\bullet s_{(2,1)}={}& \underline{\tfrac{9}{16} s_{(1,0)}+\tfrac{5}{2} s_{(2,1)}},\\
  s_{(2,0)}\bullet s_{(2,2)}={}& \underline{\tfrac{9}{16} s_{(1,1)}+\tfrac{5}{2} s_{(2,2)}},& 
  s_{(1,1)}\bullet s_{(1,1)}={}& s_{(2,2)},\\
  s_{(1,1)}\bullet s_{(2,1)}={}& \underline{\tfrac{9}{16} s_{(1,0)}},&
  s_{(1,1)}\bullet s_{(2,2)}={}& \underline{\tfrac{9}{16} s_{(2,0)}-\tfrac{5}{2} s_{(2,2)}},\\
  s_{(2,1)}\bullet s_{(2,1)}={}& \underline{\tfrac{9}{16} s_{(1,1)}+\tfrac{9}{16} s_{(2,0)}},&
  s_{(2,1)}\bullet s_{(2,2)}={}& \underline{\tfrac{9}{16} s_{(2,1)}},\\
  s_{(2,2)}\bullet s_{(2,2)}={}& \underline{\tfrac{81}{256} s_{(0,0)}+\tfrac{45}{32} s_{(1,1)}}\\
  {}& \underline{{} -\tfrac{45}{32} s_{(2,0)}+\tfrac{25}{4}s_{(2,2)}}.
\end{align*}}

\end{ex}

\begin{ex}
  In this example $\fg = \fgl_5$ and $\fk = \fgl_3 \oplus \fgl_2$. Then $|W^1| = 10$.
The change of bases matrices are: 
{\small\[ 
M = \begin{bmatrix}
  1 & 1& 1 &1 & 1& 1& 1& 1& 1& 1\\[3pt]
  3 & 2& 1 &0 & 1& 0& -1& -1& -2& -3\\[3pt]
  7 & 4& 3 &4 & 1& 1& 3& 1& 4& 7\\[3pt]
  2 & 0& -2 &-4 & 0& -1& -2& 0& 0& 2\\[3pt]
  15 & 8& 5 &0 & 1& 0& -5& -1& -8& -15\\[3pt]
  6 & 0& -2 &0 & 0& 0& 2& 0& 0& -6\\[3pt]
  14 & 0& -6 &-16 &0& -1& -6& 0& 0& 14\\[3pt]
  4 & 0& 4 &16 & 0& 1& 4& 0& 0& 4\\[3pt]
  12 & 0& 4 &0 & 0& 0& -4& 0& 0& -12\\[3pt]
  8 & 0& -8 &-64 & 0& -1& -8& 0& 0& 8\\[3pt]

\end{bmatrix}, \]\[
\quad M^{-1} = \begin{bmatrix} 
  0&0&0&\tfrac{1}{36} &0& \tfrac{1}{24}&  \tfrac{1}{72} & \tfrac{7}{144}
  &\tfrac{1}{48} &\tfrac{1}{144}\\[3pt]
  \tfrac{-1}{6}&\tfrac{-1}{12}&\tfrac{1}{6}&\tfrac{-5}{24}&\tfrac{1}{12} & 0& \tfrac{1}{24} & \tfrac{-5}{24} & \tfrac{-1}{12} &\tfrac{-1}{24}\\[3pt]
  0&0&0&\tfrac{1}{4} &0& \tfrac{-1}{8}&  \tfrac{-1}{8} & \tfrac{3}{16} &\tfrac{1}{16} &\tfrac{1}{16}\\[3pt]
  0&0&0&\tfrac{-1}{36} &0& 0&  \tfrac{1}{36} & \tfrac{-1}{36} &0 &\tfrac{-1}{36}\\[3pt]
  \tfrac{2}{3}&\tfrac{2}{3}&\tfrac{-1}{6}&\tfrac{5}{6}&\tfrac{-1}{6} & 0& \tfrac{-1}{6} & \tfrac{5}{24} & \tfrac{1}{24} &\tfrac{1}{24}\\[3pt]
  0&0&0&\tfrac{-16}{9} &0& 0&  \tfrac{4}{9} & \tfrac{-4}{9} &0 &\tfrac{-1}{9}\\[3pt]
     0&0&0&\tfrac{1}{4} &0& \tfrac{1}{8}&  \tfrac{-1}{8} & \tfrac{3}{16} &\tfrac{-1}{16} &\tfrac{1}{16}\\
  \tfrac{2}{3}&\tfrac{-2}{3}&\tfrac{-1}{6}&\tfrac{5}{6}&\tfrac{1}{6} & 0& \tfrac{-1}{6} & \tfrac{5}{24} & \tfrac{-1}{24} &\tfrac{1}{24}\\[3pt]
  \tfrac{-1}{6}&\tfrac{1}{12}&\tfrac{1}{6}&\tfrac{-5}{24}&\tfrac{-1}{12} & 0& \tfrac{1}{24} & \tfrac{-5}{24} & \tfrac{1}{12} &\tfrac{-1}{24}\\[3pt]
  0&0&0&\tfrac{1}{36} &0& \tfrac{-1}{24}&  \tfrac{1}{72} & \tfrac{7}{144} &\tfrac{-1}{48} & \tfrac{1}{144}\\[3pt]
\end{bmatrix}.
\]}
Again the new multiplication is calculated by expressing~$s_\lambda$ and~$s_\mu$ as vectors in~$\bbC^{10}$, performing the Hadamard product then applying the change of basis from~$\{\pr_i\}$ to~$\{s_\lambda \}$. 
(Again the terms vanishing in the associated graded algebra~$\twedge(\fp)^\fk$ are underlined.)
{\allowdisplaybreaks
  \small\begin{align*}
s_{(1,0)} \bullet s_{(1,0)}={}& s_{(1,1)}+s_{(2,0)},\\
s_{(1,0)} \bullet s_{(2,0)}={}& s_{(2,1)}+s_{(3,0)}, &
s_{(1,0)} \bullet s_{(1,1)}={}& s_{(2,1)},\\
s_{(1,0)} \bullet s_{(3,0)}={}& s_{(3,1)}- \underline{4 s_{(0,0)}+5 s_{(2,0)}},&
s_{(1,0)} \bullet s_{(2,1)}={}& s_{(2,2)}+s_{(3,1)},\\
s_{(1,0)} \bullet s_{(3,1)}={}& s_{(3,2)} + \underline{5 s_{(2,1)}},&
s_{(1,0)} \bullet s_{(2,2)}={}& s_{(3,2)},\\
s_{(1,0)} \bullet s_{(3,2)}={}& s_{(3,3)} + \underline{4 s_{(1,1)}+5 s_{(2,2)}},&
s_{(1,0)} \bullet s_{(3,3)}={}& \underline{4 s_{(2,1)}},\\
s_{(2,0)} \bullet s_{(2,0)}={}& s_{(2,2)}+s_{(3,1)} \underline{-4 s_{(0,0)}+5 s_{(2,0)}},&
s_{(1,1)} \bullet s_{(2,0)}={}& s_{(3,1)},\\
s_{(2,0)} \bullet s_{(3,0)}={}& s_{(3,2)} \underline{-4 s_{(1,0)}+5 s_{(2,1)}+5 s_{(3,0)}},&
s_{(2,0)} \bullet s_{(2,1)}={}& s_{(3,2)} + \underline{5 s_{(2,1)}},\\
s_{(2,0)} \bullet s_{(3,1)}={}& s_{(3,3)} + \underline{5 s_{(2,2)}+5 s_{(3,1)}},&
s_{(2,0)} \bullet s_{(2,2)}={}& \underline{4 s_{(1,1)}+5 s_{(2,2)}},\\
s_{(2,0)} \bullet s_{(3,2)}={}& \underline{4 s_{(2,1)}+5 s_{(3,2)}},&
s_{(2,0)} \bullet s_{(3,3)}={}& \underline{4 s_{(2,2)}+5 s_{(3,3)}},
\end{align*}                                                                          
\begin{align*}
s_{(1,1)} \bullet s_{(1,1)}={}& s_{(2,2)},\\
s_{(1,1)} \bullet s_{(3,0)}={}& \underline{5 s_{(2,1)}},&
s_{(1,1)} \bullet s_{(2,1)}={}& s_{(3,2)},\\
s_{(1,1)} \bullet s_{(3,1)}={}& \underline{4 s_{(1,1)}+5 s_{(2,2)}},&
s_{(1,1)} \bullet s_{(2,2)}={}& s_{(3,3)},\\
s_{(1,1)} \bullet s_{(3,2)}={}& \underline{4 s_{(2,1)}},&
s_{(1,1)} \bullet s_{(3,3)}={}& \underline{4 s_{(3,1)}-5 s_{(3,3)}},\\
s_{(3,0)} \bullet s_{(3,0)}={}& s_{(3,3)} \underline{-20 s_{(0,0)}+21 s_{(2,0)}}&
s_{(2,1)} \bullet s_{(3,0)}={}& \underline{5 s_{(2,2)}+5 s_{(3,1)}},\\
                          {}& \underline{+5 s_{(2,2)}+5 s_{(3,1)}},&\\
s_{(3,0)} \bullet s_{(3,1)}={}& \underline{25 s_{(2,1)}+5 s_{(3,2)}},&
s_{(2,2)} \bullet s_{(3,0)}={}& \underline{5 s_{(3,2)}},\\
s_{(3,0)} \bullet s_{(3,2)}={}& \underline{20 s_{(1,1)}+25 s_{(2,2)}+5 s_{(3,3)}},&
s_{(3,0)} \bullet s_{(3,3)}={}& \underline{20 s_{(2,1)}},\\
s_{(2,1)} \bullet s_{(2,1)}={}& s_{(3,3)} + \underline{4 s_{(1,1)}+5 s_{(2,2)}},&
s_{(2,1)} \bullet s_{(3,1)}={}& \underline{4 s_{(2,1)}+5 s_{(3,2)}},\\
s_{(2,1)} \bullet s_{(2,2)}={}& \underline{4 s_{(2,1)}},&
s_{(2,1)} \bullet s_{(3,2)}={}& \underline{4 s_{(2,2)}+4 s_{(3,1)}},\\
s_{(2,1)} \bullet s_{(3,3)}={}& \underline{4 s_{(3,2)}},&
s_{(3,1)} \bullet s_{(3,1)}={}& \underline{20 s_{(1,1)}+25 s_{(2,2)}+4 s_{(3,1)}},
\end{align*}
\begin{align*}
s_{(2,2)} \bullet s_{(3,1)}={}& \underline{4 s_{(2,2)}+5 s_{(3,3)}},\\
s_{(3,1)} \bullet s_{(3,2)}={}& \underline{20 s_{(2,1)}+4 s_{(3,2)}},\\
s_{(3,1)} \bullet s_{(3,3)}={}& \underline{20 s_{(3,1)}-21 s_{(3,3)}},\\
s_{(2,2)} \bullet s_{(2,2)}={}& \underline{4 s_{(3,1)}-5 s_{(3,3)}},\\
s_{(2,2)} \bullet s_{(3,2)}={}& \underline{4 s_{(3,2)}},\\
s_{(2,2)} \bullet s_{(3,3)}={}& \underline{16 s_{(1,1)}+20 s_{(2,2)}-20 s_{(3,1)}+25 s_{(3,3)}},\\
s_{(3,2)} \bullet s_{(3,2)}={}& \underline{16 s_{(1,1)}+20 s_{(2,2)}+4 s_{(3,3)}},\\
s_{(3,2)} \bullet s_{(3,3)}={}& \underline{16 s_{(2,1)}},\\
s_{(3,3)} \bullet s_{(3,3)}={}& \underline{-80 s_{(1,1)}-84 s_{(2,2)}+100 s_{(3,1)}-105 s_{(3,3)}}.
\end{align*}}

\end{ex}

In both examples, as well as in general, deleting the underlined terms gives the Littlewood-Richardson coefficients.

\providecommand{\bysame}{\leavevmode\hbox to3em{\hrulefill}\thinspace}
\providecommand{\MR}{\relax\ifhmode\unskip\space\fi MR }
\providecommand{\MRhref}[2]{%
  \href{http://www.ams.org/mathscinet-getitem?mr=#1}{#2}
}
\providecommand{\href}[2]{#2}


\end{document}